\documentclass{amsart}

\usepackage{amssymb,amsmath,color,ytableau,graphicx,float}
\usepackage{amsthm}
\usepackage{url}
\usepackage{tikz}

\definecolor{red}{rgb}{1,0,0}
\definecolor{blue}{rgb}{.2,.2,.8}
\newcommand{\ccom}[1]{{\color{teal}{CB: #1}} }

\def\eu{\textrm{eu}}
\def\ed{\textrm{ed}}
\def\ou{\textrm{ou}}
\def\od{\textrm{od}}
\def\on{\textrm{ond}}
\def\en{\textrm{end}}

\newtheorem{theorem}{Theorem}[section]
\newtheorem{corollary}[theorem]{Corollary}
\newtheorem{proposition}{Proposition}
\newtheorem{conjecture}{Conjecture}

\theoremstyle{definition}

\newtheorem{example}{Example}
\newtheorem{remark}{Remark}

\newcommand{\ds}{\displaystyle}

\begin{document}

\title[Inequalities for PSP-partition numbers]{Combinatorial proofs of inequalities involving the number of partitions with parts separated by parity}
\author{Cristina Ballantine}\address{Department of Mathematics and Computer Science\\ College of the Holy Cross \\ Worcester, MA 01610, USA \\} 
\email{cballant@holycross.edu} 
\author{Amanda Welch} \address{Department of Mathematics and Computer Science\\ Eastern Illinois University \\ Charleston, IL 61920, USA \\} \email{arwelch@eiu.edu} 

\begin{abstract} We consider the number of various partitions of $n$ with parts separated by parity and prove combinatorially several inequalities between these numbers. For example, we show that for $n\geq 5$ we have $p_{\od}^{\eu}(n)<p_{\ed}^{\ou}(n)$, where $p_{\od}^{\eu}(n)$ is the number of partitions of $n$ with odd parts distinct and even parts unrestricted and all odd parts less than all even parts and $p_{\ed}^{\ou}(n)$ is the number of partitions of $n$ with even parts distinct and odd parts unrestricted and all even parts less than all odd parts. We also prove a conjectural inequality of Fu and Tang involving partitions with parts separated by parity with restrictions on the multiplicity of parts. 
\end{abstract}

\maketitle



\section{Introduction}

A partition of a non-negative integer $n$ is a non-increasing sequence of  positive integers that sum  to $n$. We write a partition of $n$  as $\lambda = (\lambda_1, \lambda_2, \ldots, \lambda_{\ell})$ with $\lambda_1\geq \lambda_2\geq \cdots \geq \lambda_\ell$ and $\lambda_1 + \lambda_2 + \dotsm + \lambda_{\ell} = n$. We refer to the terms $\lambda_i$ as the parts of $\lambda$. 

In this article, we consider the number of partitions of $n$ with parts separated by parity. These partitions were introduced by Andrews in \cite{A19, A18}.  Because the numbers of partitions of $n$ with parts separated by parity (in different combinations) display modular properties, several authors have studied  different aspects of their behavior, see for example  \cite{BCN, BD, BE, BJS, CC, C, C2021, FT, GJ, MSS, P, RB,  RS}.  Following the notation in \cite{A19}, we denote by $\mathcal P_{\textrm{yz}}^{\textrm{wx}}(n)$ the set of partitions of $n$ with parts of type $w$ satisfying condition $x$ or of type $y$ satisfying 
condition $z$ and all parts of type $w$ satisfying condition $x$ are larger than all parts  of type $y$ satisfying 
condition $z$. Here, the symbols $w$ and $y$ satisfy $w\neq y$ and  are $\textrm e$ (for even) or $\textrm o$ (for odd).  The symbols $x$ and $y$ are  $\textrm u$ (for unrestricted) or $\textrm d$  (for distinct). For example, $\mathcal P_{\ou}^{\ed}(n)$ is the set of partitions of $n$ in which all even parts are greater than all odd parts, odd parts are unrestricted and even parts are distinct. We set  $p_{\textrm{yz}}^{\textrm{wx}}(n):=|\mathcal P_{\textrm{yz}}^{\textrm{wx}}(n)|$. 

In this article, we allow partitions in $\mathcal P_{yz}^{wx}(n)$ to have no parts of type $y$ or no parts of type $w$. 

\begin{example}
    Let $n = 6$. Then 
    $$\mathcal P_{\ou}^{\ed}(6) = \{(6), (5, 1), (4, 2), (4, 1^2), (3^2), (3, 1^3), (2, 1^5), (1^7)\}.$$ Thus, $p_{\ou}^{\ed}(6) = 8$. 
\end{example}

Bringmann, Craig and Nazaroglu \cite{BCN} studied the asymptotic behavior of the eight sequences $p_{yz}^{wx}(n)$ (for all possible combinations of $w,x,y,z$). At the end of their article they conclude that for sufficiently large $n$, \begin{equation}\label{chain} p_{\ed}^{\od}(n) <p^{\ed}_{\od}(n) <p^{\eu}_{\od}(n)  <p_{\eu}^{\od}(n)
<p_{\ed}^{\ou}(n)<p_{\eu}^{\ou}(n) <p^{\ed}_{\ou}(n) <p^{\eu}_{\ou}(n).\end{equation}

The second, fifth, and last inequalities are immediate since $\mathcal P^{\ed}_{\od}(n) \subseteq \mathcal P^{\eu}_{\od}(n)$, $\mathcal P_{\ed}^{\ou}(n)\subseteq \mathcal P_{\eu}^{\ou}(n)$ and $\mathcal P^{\ed}_{\ou}(n) \subseteq \mathcal P^{\eu}_{\ou}(n)$. Bringmann, Craig and Nazaroglu asked for combinatorial proofs of the other inequalities. 

In section \ref{ineq}, we give combinatorial injections proving the following inequalities:\begin{align}\label{i1} p_{\ed}^{\od}(n) &<p^{\ed}_{\od}(n), \ \ \text{for } n\geq 11;\\ \label{i2} p_{\eu}^{\ou}(n) & <p_{\ou}^{\eu}(n), \ \ \text{for } n\geq 3; \\ \label{i3} p_{\od}^{\eu}(n) & <p_{\ed}^{\ou}(n), \ \ \text{for } n\geq 5; \\ \label{i4} p^{\od}_{\eu}(n) & <p^{\ed}_{\ou}(n), \ \ \text{for } n\geq 2;\\ \label{i5} p_{\od}^{\ed}(n) & <p_{\eu}^{\od}(n), \ \ \text{for } n\geq 8. \end{align}

In \cite{FT}, Fu and Tang considered partitions with parts separated by parity with additional conditions on the multiplicity of parts. We adapt their notation to match the notation of this article. 
    Let $\overline{\mathcal P}_{\eu}^{\ou}(n)$ be the subset of partitions in $\mathcal P_{\eu}^{\ou}(n)$ in which the largest even part appears with odd multiplicity and all other parts appear with even multiplicity. Note that a partition of $n$  with only odd parts is considered to be in  $\overline{\mathcal P}_{\eu}^{\ou}(n)$ with a single part equal to $0$ (and multiplicity one). We  set $\overline p_{\eu}^{\ou}(n):=|\overline{\mathcal P}_{\eu}^{\ou}(n)|$. 
Similarly, let $\overline{\mathcal P}^{\eu}_{\ou}(n)$ be the subset of partitions in $\mathcal P^{\eu}_{\ou}(n)$ in which both even and odd parts occur, the largest even part and the largest odd part each appear with odd multiplicity and all other parts appear with even multiplicity. We  set $\overline p^{\eu}_{\ou}(n):=|\overline{\mathcal P}^{\eu}_{\ou}(n)|$.
Fu and Tang made the following conjecture. \begin{conjecture} \label{conj1} Let  $n\geq 3$. Then,   $$\overline p_{\eu}^{\ou}(2n)<\overline p^{\eu}_{\ou}(2n+1).$$\end{conjecture} We prove the conjecture in section \ref{sec_conj}.

\section{Definitions and Notation} Throughout the article we use the following notation. Given a partition $\lambda$, the length of $\lambda$, denoted $\ell(\lambda)$,  is the number of parts in  $\lambda$. 
We denote by $\lambda^o$, respectively $\lambda^e$,  the partition whose parts are precisely the odd, respectively even,  parts of $\lambda$. We denote by $\ell_o(\lambda)$, respectively $\ell_e(\lambda)$, the number of odd, respectively even,  parts in $\lambda$, and when $\lambda$ contains both even and odd parts, $\ell_m(\lambda)$ denotes the minimum of $\ell_o(\lambda)$ and $\ell_e(\lambda)$. Moreover, $\lambda_s^o$, respectively $\lambda_s^e$, denotes the last part in $\lambda^o$, respectively $\lambda^e$. By $\lambda^o_1$, respectively $\lambda^e_1$, we mean the first part in $\lambda^o$, respectively $\lambda^e$.  If $j$ is a part of $\lambda$, we denote by $a_{>j}$  the smallest part in $\lambda$ that is larger than $j$, if such a part exists. We write $m_\lambda(j)$ for the multiplicity of $j$ in $\lambda$, i.e., the number of times $j$ occurs as a part in $\lambda$.  We also write  $\ell_{>j}(\lambda)$, respectively $\ell_{\geq j}(\lambda)$, for the number of parts greater than, respectively greater than or equal to,  $j$ in $\lambda$. Occasionally, we use the frequency notation for partitions and write $\lambda=(i_1^{m_1}, i_2^{m_2}, \ldots, i_k^{m_k})$, where  $m_j=m_\lambda(i_j)$.  For convenience, if $j>\ell(\lambda)$, we set $\lambda_j=0$. The set of  partitions of $n$ with distinct parts is denoted by  $\mathcal D(n)$.  

If $\mathcal A(n)$ is the set of partitions of $n$ satisfying certain properties, we define $$\mathcal A: = \bigcup_{n\geq0}\mathcal A(n).$$

We identify partitions with their multiset of parts. If $a$ is a positive integer, we write $a\in \lambda$ to mean that $a$ is a part of $\lambda$.  If   $\lambda$ and $\mu$ are partitions, by $\lambda\cup \mu$ and $\lambda\setminus \mu$ we mean the obvious operations on the multisets of parts of $\lambda$ and $\mu$. Moreover, $\lambda\setminus \mu$ is only defined if $\mu \subseteq \lambda$ as multisets. 

The {Ferrers diagram} of a partition $\lambda=(\lambda_1, \lambda_2, \ldots, \lambda_{\ell(\lambda)})$ is an array of left justified boxes such that the $i$th row from the top contains $\lambda_i$ boxes. We abuse notation and use $\lambda$ to mean a partition or its Ferrers diagram. 
\begin{example} The Ferrers diagram of $\lambda=(5,4, 3, 3, 2)$  is shown below.  \medskip

 \begin{center} \ytableausetup{smalltableaux}{\ydiagram[*(white)]
{5,4,3,3,2}}\end{center}
 \end{example}

\section{Combinatorial proofs of inequalities \eqref{i1} --  \eqref{i5}} \label{ineq}

\subsection{The inequality $p_{\ed}^{\od}(n) <p^{\ed}_{\od}(n)$} \label{s1}

Let $n\geq 0$. We create an injection $$\varphi_1:\mathcal P_{\ed}^{\od}(n)\to\mathcal P^{\ed}_{\od}(n).$$ Let $\lambda\in \mathcal P_{\ed}^{\od}(n).$ If $\lambda^o=\emptyset$ or $\lambda^e=\emptyset$, define $\varphi_1(\lambda):=\lambda$. If $\lambda^e, \lambda^o\neq \emptyset$, define $\varphi_1(\lambda)$ to be the partition obtained from $\lambda$ by subtracting one from each of the last $\ell_m(\lambda)$ parts of  $\lambda$ and adding one to each of the first $\ell_m(\lambda)$ parts of  $\lambda$. In each case, the partition $\varphi_1(\lambda)\in \mathcal P^{\ed}_{\od}(n)$. Moreover, if $\lambda^o, \lambda^e\neq \emptyset$, the smallest even part of  $\varphi_1(\lambda)$ differs by at least three from the largest odd part of $\varphi_1(\lambda)$. 

If we denote by  $\widetilde{\mathcal P}^{\ed}_{\od}(n)$ the subset of  partitions $\mu\in \mathcal P^{\ed}_{\od}(n)$ with $\mu^o=\emptyset$ or $\mu^e=\emptyset$, or $\mu^o, \mu^e\neq \emptyset$ and $\mu^e_s-\mu^o_1\geq 3$, then $\varphi_1: \mathcal P_{\ed}^{\od}(n) \to \widetilde{\mathcal P}^{\ed}_{\od}(n)$ is a bijection. To see that $\varphi_1$ is invertible, let $\mu \in \widetilde{\mathcal P}^{\ed}_{\od}(n)$. If $\mu^o=\emptyset$ or $\mu^e=\emptyset$, then $\varphi_1^{-1}(\mu)=\mu$. If $\mu^o, \mu^e\neq \emptyset$, then $\varphi_1^{-1}(\mu)$ is the partition obtained from $\mu$ by subtracting one from each of the first $\ell_m(\mu)$ parts of $\mu$ and adding one to each of the last $\ell_m(\mu)$ parts of $\mu$.

\begin{example}
Let $n = 35$ and $\lambda = (17, 12, 6)$. Then $\ell_m(\lambda) = \ell_o(\lambda)=1$. Thus, $\varphi_1(\lambda) = (18, 12, 5)$.
\end{example}

We have the following interpretation of  $p^{\ed}_{\od}(n)-p_{\ed}^{\od}(n)$.

\begin{proposition} \label{P1} Let $n\geq 0$. Then, $p^{\ed}_{\od}(n)-p_{\ed}^{\od}(n)$ equals the number of partitions $\lambda \in \mathcal P^{\ed}_{\od}(n)$ with $\lambda^o, \lambda^e\neq \emptyset$ and   $\lambda^e_s-\lambda^o_1=1$.
\end{proposition}

 One can verify directly that there is no partition $\lambda \in \mathcal P^{\ed}_{\od}(10)$ with both even and odd parts such that $\lambda^e_s-\lambda^o_1=1$. If $n\geq 11$, there is a partition $\lambda \in \mathcal P^{\ed}_{\od}(n)$ with both even and odd parts such that $\lambda^e_s-\lambda^o_1=1$. For example, 
\begin{itemize}
\item[(i)] if $n=4k$, $k\geq 3$, then $\lambda=(2k, 2k-1, 1)$;

\item[(ii)] if $n=13$, then $\lambda=(6,4,3)$; 
\\ if $n=17$, then $\lambda=(10, 4,3)$; 
\\ if $n=4k+1$, $k\geq 5$, then $\lambda=(2k-2, 2k-3, 5, 1)$;

\item[(iii)] if $n=4k+2$, $k\geq 3$, then $\lambda=(2k, 2k-1, 3)$;

\item[(iv)] if $n=4k+3$, $k\geq 2$, then $\lambda=(2k+2, 2k+1)$.
\end{itemize}

The combinatorial injection $\varphi_1$ together with the discussion  above proves the following theorem. 
\begin{theorem}\label{T1} The inequality $p_{\ed}^{\od}(n) \leq p^{\ed}_{\od}(n)$ holds for all $n\geq 0$. Moreover, $p_{\ed}^{\od}(n) < p^{\ed}_{\od}(n)$ for $n\geq 11$.
\end{theorem}

Inclusion shows that $p^{\ed}_{\od}(n)\leq p_{\ou}^{\ed}(n)$ for $n\geq 0$. Since for $n\geq 2$ we have  $(1^n)\in \mathcal P_{\ou}^{\ed}(n)\setminus \mathcal P^{\ed}_{\od}(n)$, we have $p^{\ed}_{\od}(n)< p_{\ou}^{\ed}(n)$ for $n\geq 2$. Hence, Theorem \ref{T1} gives the following inequality.  

\begin{corollary}  If $n\geq 2$, then  $p^{\od}_{\ed}(n) <p_{\ou}^{\ed}(n)$.
\end{corollary}

\smallskip

\subsection{The inequality $p_{\eu}^{\ou}(n) <p^{\eu}_{\ou}(n)$}

Let $n\geq 0$. The mapping $\varphi_1$ defined in the previous section can be extended to $\mathcal P_{\eu}^{\ou}(n)$. It becomes an injection $$\varphi_1:\mathcal P_{\eu}^{\ou}(n)\to\mathcal P^{\eu}_{\ou}(n).$$ We denote by $\widetilde{\mathcal P}^{\eu}_{\ou}(n)$ the subset of partitions $\mu\in \mathcal P^{\eu}_{\ou}(n)$ satisfying one of the following conditions
\begin{itemize} \item[(i)] $\mu^o=\emptyset$;
\item [(ii)] $\mu^e=\emptyset$;
\item [(iii)] $\mu^o, \mu^e\neq \emptyset$, $\ell_o(\mu)\leq \ell_e(\mu)$ and $\mu_{\ell_o(\mu)}-\mu_{\ell_o(\mu)+1}\geq 2$;
\item [(iv)] $\mu^o, \mu^e\neq \emptyset$, $\ell_e(\mu)< \ell_o(\mu)$ and $\mu_{\ell(\mu)-\ell_e(\mu)}-\mu_{\ell(\mu)-\ell_e(\mu)+1}\geq 2$.
\end{itemize} Then $\varphi_1:\mathcal P_{\eu}^{\ou}(n) \to \widetilde{\mathcal P}^{\eu}_{\ou}(n)$ is a bijection. To see that $\varphi_1$ is invertible, let $\mu \in \widetilde{\mathcal P}^{\eu}_{\ou}(n)$. If $\mu^o=\emptyset$ or $\mu^e=\emptyset$, then $\varphi_1^{-1}(\mu)=\mu$. If $\mu^o, \mu^e\neq \emptyset$, then $\varphi_1^{-1}(\mu)$ is the partition obtained from $\mu$ by subtracting one from each of the first $\ell_m(\mu)$ parts of $\mu$ and adding one to each of the last $\ell_m(\mu)$ parts of $\mu$.

\begin{example}
Let $n = 43$ and $\lambda = (11, 9^2, 6, 4, 2^2)$. Then $\ell_m(\lambda) = \ell_o(\lambda)=3$. Thus, $\varphi_1(\lambda) = \mu = (12, 10^2, 6, 3, 1^2)$. We note that this falls under Case (iii) as $\ell_e(\mu) > \ell_o(\mu)$ and $\mu_{\ell_o(\mu)}-\mu_{\ell_o(\mu)+1} = \mu_3 - \mu_4 = 10 - 6 = 4.$

Similarly, let $n = 44$ and $\lambda = (11, 9^2, 7, 4, 2^2)$. Then $\ell_m(\lambda) = \ell_e(\lambda)=3$. Thus, $\varphi_1(\lambda) = \mu = (12, 10^2, 7, 3, 1^2)$. We note that this falls under Case (iv) as $\ell_e(\mu) < \ell_o(\mu)$ and $\mu_{\ell(\mu)-\ell_e(\mu)}-\mu_{\ell(\mu)-\ell_e(\mu)+1} = \mu_4 - \mu_5 = 7 - 3 = 4.$
\end{example}

\begin{remark} We could interpret the excess $p^{\eu}_{\ou}(n)-p_{\eu}^{\ou}(n)$ in the manner of Proposition \ref{P1}, but the description is less elegant. 
\end{remark}

One can verify directly that $p^{\eu}_{\ou}(2)=p_{\eu}^{\ou}(2)=2$. If $n\geq 3$, the partition $(2, 1^{n-2})\in \mathcal P^{\eu}_{\ou}(n)$ is not in $\widetilde{\mathcal P}^{\eu}_{\ou}(n)$. Together with the combinatorial injection $\varphi_1$, this  proves the following theorem.

\begin{theorem} The inequality $p_{\eu}^{\ou}(n) \leq p^{\eu}_{\ou}(n)$ holds for all $n\geq 0$.  Moreover,  $p_{\eu}^{\ou}(n) <p^{\eu}_{\ou}(n)$ for $n\geq 3$.
    \end{theorem}
\smallskip

\subsection{The inequality $p_{\od}^{\eu}(n) <p_{\ed}^{\ou}(n)$}

 In this section, given a partition  $\lambda\not \in \mathcal D$,  we denote by $\lambda_r$ the smallest part size in $\lambda$ with multiplicity greater than one; we also denote by $\lambda_{rr}$ the second smallest part size of $\lambda$ with multiplicity greater than one, if such a part exists. 

\medskip

Let $n\geq 0$. We create an injection $$\varphi_2:\mathcal P_{\od}^{\eu}(n)\to\mathcal P^{\ou}_{\ed}(n).$$ 

Let $\lambda\in \mathcal P_{\od}^{\eu}(n)$. We consider five cases.  For $1\leq i\leq 5$, we denote by $\mathcal B_i(n)$ the image  of the set of all partitions in Case $i$ under $\varphi_2$.  At the end of this section, we give examples for each case.\smallskip

\noindent \underline{Case 1:} If $\lambda^e=\emptyset$, define $\varphi_2(\lambda):=\lambda$. Then, $$\mathcal B_1(n)=\{\mu\in \mathcal P^{\ou}_{\ed}(n)\ \big| \ \mu=\mu^o, \, \mu\in \mathcal D\}.$$
\smallskip 

\noindent \underline{Case 2:} ($n$ even) If $\lambda^o=\emptyset$ and $\ell_e(\lambda) = 1$, define $\varphi_2(\lambda):=\lambda$. Then, $$\mathcal B_2(n)=\{(n)\}.$$

\smallskip

\noindent \underline{Case 3:} If $\ell_o(\lambda) \geq 2$ and $\ell_e(\lambda) = 1$,  define $$\varphi_2(\lambda):=(\lambda^e_1-1, \lambda^o_1, \ldots, \lambda^o_{s-1}, \lambda^o_s+1),$$ 
 the partition obtained from $\lambda$ by subtracting one from the even part of $\lambda$ and adding one to the last part of $\lambda$. Then, $$\mathcal B_3(n)=\{\mu\in \mathcal P^{\ou}_{\ed}(n)\ \big| \  \ell_e(\mu)=1, \ell_o(\mu) \geq 2,\, \mu^o_1\geq\mu^o_2>\mu^o_3>\ldots >\mu^o_s>\mu^e_1  \}.$$
 
\smallskip

\noindent \underline{Case 4:} ($n$ even)
  If $\ell_e(\lambda) \geq 2$ and $\lambda^o = (\lambda^e_s - 1, 1)$,   define 
  $$\varphi_2(\lambda):=(\lambda^e_1+1, \lambda^e_2-1, \ldots, \lambda^e_s-1, \lambda_1^o=\lambda^e_s-1)\cup (1^{\ell_e(\lambda)-1}),$$  the partition obtained from $\lambda$ by removing a part equal to $1$, adding one to the largest part of $\lambda$,  subtracting one from every even part  of $\lambda$ other than $\lambda^e_1$ and inserting $\ell_e(\lambda)-1$ parts equal to $1$.  Then, 
$$\mathcal B_4(n)=\left\{\mu\in \mathcal P^{\ou}_{\ed}(n)\left|\begin{array}{l}\mu=\mu^o, \, \mu\not\in \mathcal D, \, \ell(\mu)\geq 4,\, \ell(\mu) \text{ even}, \\ \ \\ \mu_1\neq \mu_2,m_\mu(1)=\ell_{>1}(\mu)-2, \, m_\mu(a_{>1})>1\end{array} \right.\right\}.$$

  \smallskip

  \smallskip
  
\noindent \underline{Case 5:} If $\ell_e(\lambda)=\ell_o(\lambda)=1$, or $\ell_e(\lambda)\geq 2$ and $\lambda^o \neq (\lambda^e_s - 1, 1)$, 
let $$k:=\begin{cases}\min\{\lambda_1^o ,  \lambda_s^e - \lambda_1^o\} & \text{ if } \lambda^o \neq \emptyset,\\ 1 & \text{ if } \lambda^o = \emptyset.\end{cases}$$  and define $$\varphi_2(\lambda):=(\lambda^e_1-k, \lambda^e_2-k, \ldots , \lambda^e_s-k)\cup \lambda^o\cup (k^{\ell_e(\lambda)}),$$  the partition obtained from $\lambda$ by subtracting $k$ from each part of $\lambda^e$ and inserting $\ell_e(\lambda)$ parts equal to $k$. 

Then, $\mathcal B_5(n)$ is the subset of partitions
$\mu\in \mathcal P^{\ou}_{\ed}(n)$ with $\mu=\mu^o$, $\mu\not\in \mathcal D$,  $\ell(\mu)\geq 3$, and if $\ell(\mu)>3$ then
\begin{itemize}
    \item[(I)] If $\mu_r=1$ and $m_\mu(1)<\ell_{>1}(\mu)$, then $\mu_{rr}$  exists and $m_\mu(1)\in \{\ell_{\geq \mu_{rr}}(\mu)-1, \, \ell_{\geq \mu_{rr}}(\mu)\}$.  
\item[(II)] If $\mu_r>1$, then $\ell(\mu)\geq 5$ and
\begin{itemize}
    \item[(i)] if  $m_\mu(\mu_r)<\ell_{>\mu_r}(\mu)$, then $\mu_{rr}$  exists  and $m_\mu(\mu_r)\in \{\ell_{\geq \mu_{rr}}(\mu)-1, \, \ell_{\geq \mu_{rr}}(\mu)\};$
    \item[(ii)] if $m_\mu(\mu_r)=\ell_{>\mu_r}(\mu)$, then $m_\mu(a_{>\mu_r})>1;$ 
    \item[(iii)] if $m_\mu(\mu_r)>\ell_{>\mu_r}(\mu)$, then $\ell_{\geq \mu_r}(\mu$) is odd. 
    \end{itemize}
\end{itemize}
{\bf Note:} If $\mu_r=1$ and $m_\mu(1)\geq \ell_{>1}(\mu)$, there are no additional conditions.

Examples 6. -- 15. below illustrate different partitions that belong to Case 5 and the discussion before the various examples clarifies  the description of $\mathcal B_5(n)$.

The sets $\mathcal B_i(n)$, $1\leq i\leq 5$, are mutually disjoint. {To see that $\mathcal B_4(n)\cap \mathcal B_5(n)=\emptyset$, let $\mu\in \mathcal B_5(n)$ and consider the following cases: 
\begin{itemize}
    \item if  $\mu_r>1$, then $\ell(\mu)\geq 5$ and $m_\mu(1)=1$ and thus $m_\mu(1)\neq \ell_{>1}(\mu)-2$;
    \item if  $\mu_r=1$ and   $m_\mu(1)= \ell_{>1}(\mu)-2$, then the condition $m_\mu(1)\in \{\ell_{\geq \mu_{rr}}(\mu)-1, \, \ell_{\geq \mu_{rr}}(\mu)\}$ implies that  $m_\mu(a_{>1})=1$. 
\end{itemize}  

\noindent {\bf Note:} In fact, if $\mu\in \mathcal B_5(n)$ with $\mu_r=1$ and   $m_\mu(1)\leq \ell_{>1}(\mu)-2$,  then   $m_\mu(a_{>1})=1$.}\smallskip

The mapping  $$\varphi_2:\mathcal P_{\od}^{\eu}(n) \to \displaystyle \bigcup_{i=1}^5\mathcal B_i(n)$$ is a bijection. To see that $\varphi_2$ is invertible, let $\mu\in \ds\bigcup_{i=1}^5\mathcal B_i(n)$. \smallskip

If $\ds \mu\in \mathcal B_1(n)\cup \mathcal B_2(n)$, then $$\varphi_2^{-1}(\mu)=\mu.$$

If $\mu\in \mathcal B_3(n)$, then $\mu$ has one even part and at least two odd parts and only the first two parts may be equal. We have $$\varphi_2^{-1}(\mu)=(\mu_1+1, \mu_2, \ldots, \mu_{\ell(\mu)-1}, \mu_{\ell(\mu)}-1).$$ 

If $n$ is even and $\mu\in \mathcal B_4(n)$, then all parts of $\mu$ are odd and $\mu_1\neq \mu_2$.  If we set $m=m_\mu(1)$, we have $\ell(\mu)=2m+2$ and $\mu_{m+1}=\mu_{m+2}>1$. Then,    $$\varphi_2^{-1}(\mu)=(\mu_1-1, \mu_2+1, \ldots, \mu_{m+1}+1, \mu_{m+2}, 1).$$

If $\mu\in \mathcal B_5(n)$, then we have the following cases. \smallskip 

If $\ell(\mu)=3$, we have $\varphi_2^{-1}(\mu)=(\mu_1+\mu_3, \mu_2)$. \smallskip

If $\ell(\mu)>3$, then

\begin{itemize}
\item if $m_\mu(\mu_r)<\ell_{>\mu_r}(\mu)$, or $m_\mu(\mu_r)=\ell_{>\mu_r}(\mu)$ and $\mu_r\neq 1$, $\varphi_2^{-1}(\mu)$ is obtained from $\mu$ by adding $\mu_r$ to each of the first $\ell_{\geq\mu_{rr}(\mu)}-1$ parts and removing  $\ell_{\geq\mu_{rr}(\mu)}-1$ parts equal to $\mu_r$. \vspace*{.1in}

\item if $m_\mu(\mu_r)> \ell_{>\mu_r}(\mu)$, or $m_\mu(\mu_r)=\ell_{>\mu_r}(\mu)$ and $\mu_r= 1$, $\varphi_2^{-1}(\mu)$ is obtained from $\mu$ by adding $\mu_r$ to each of the first $\lfloor\ell_{\geq \mu_r}(\mu)/2\rfloor$ parts of $\mu$ and removing $\lfloor\ell_{\geq \mu_r}(\mu)/2\rfloor$ parts equal to $\mu_r$. 
\end{itemize}\smallskip

One can verify directly that $p^{\eu}_{\od}(4)=p^{\ou}_{\ed}(4)=3$. If $n\geq 5$, the partition $(n-2,2)\in \mathcal P^{\ou}_{\ed} $ is not in $\bigcup_{i=1}^5 \mathcal B_i(n)$. Together with the combinatorial injection $\varphi_2$, this  proves the following theorem.

\begin{theorem} The inequality $p^{\eu}_{\od}(n)\leq p^{\ou}_{\ed}(n)$ holds for all $n\geq 0$.  Moreover,  $p^{\eu}_{\od}(n)< p^{\ou}_{\ed}(n)$ for $n\geq 5$.
    \end{theorem}

\subsubsection{Examples}

To help clarify the map $\varphi_2$ and its image, we provide examples for all the cases. We also discuss Case 5 in more detail.

\begin{example}\label{B1-4}
Let $n = 16$. We give examples for partitions in $\mathcal P_{\od}^{\eu}(16)$ in each of the Cases 1 -- 4. 
\begin{itemize}
\item[\underline{Case 1:}] If $\lambda = (15,1)$, then $\varphi_2(\lambda) = (15, 1) \in \mathcal{B}_1(16).$\smallskip 

\item[\underline{Case 2:}] If $\lambda = (16),$ then $ \varphi_2(\lambda) = (16) \in \mathcal{B}_2(16).$\smallskip

\item[\underline{Case 3:}] If $ \lambda = (8, 5, 3)$, then $\varphi_2(\lambda) = (7, 5, 4) \in \mathcal{B}_3(16).$\smallskip

\item[\underline{Case 4:}] If $\lambda = (8, 4, 3, 1)$, then $\varphi_2(\lambda) = (9, 3, 3, 1) \in \mathcal{B}_4(16)$.
\end{itemize}
\end{example}

First, we consider partitions $\lambda$ with $\ell_e(\lambda)=\ell_o(\lambda)=1$. If $\lambda = (\lambda^e_1, \lambda^o_1),$  then $\varphi_2(\lambda) = (\lambda^e_1 - k, \lambda^o_1) \cup (k) \in \mathcal B_5(n)$, as there are exactly three parts, they are all odd, and either $\lambda^e_1 - k =  \lambda^o_1$ or $k =  \lambda^o_1$ (or both  if $\lambda^e_1=2\lambda^o_1$).  This leads to  all partitions $\mu\in \mathcal B_5(n)$ with  $\ell(\mu) = 3$. Below, we consider one such example. 

\begin{example}\label{B5a}
Let $n = 13$ and $\lambda = (8, 5)$. Then $k=3$ and $\varphi_2(\lambda) = (5, 5, 3) \in \mathcal{B}_5(13)$.
\end{example}

Next, we restrict our attention to partitions with $k = 1.$ We start with partitions $\lambda$ with  $\lambda = \lambda^e = (\lambda^e_1, \lambda^e_2, \dotsc, \lambda^e_s)$ and $\ell(\lambda) > 1$.  Then $k=1$ and  $\mu=\varphi_2(\lambda) = (\lambda^e_1 - 1, \lambda^e_2 -1, \dotsc, \lambda^e_s - 1) \cup (1^{\ell_e(\lambda)}) \in \mathcal B_5(n)$, as $\mu_r = 1$ and $m_\mu(1)\geq \ell_{> 1}(\mu)$. Note that in this case  $\ell(\mu)$ is even and $m_\mu(1) = \ell_{> 1}(\lambda) + j$ with $j \geq 0$ even. Below, we give a specific example for a partition with all  parts  even and length greater than one. 

\begin{example}\label{B5b}
Let $n = 16$ and $\lambda = (8, 6, 2)$. Then $\varphi_2(\lambda) = (7, 5, 1^4) \in \mathcal{B}_5(16)$.
\end{example}

Still restricting to partitions with $k = 1$, it remains to consider partitions $\lambda$ in Case 5 satisfying $\ell_e(\lambda) \geq 2, \ell_o(\lambda) \geq 1$, and $\lambda^o \neq (\lambda^e_s - 1, 1)$ for which $k = 1$. We have two subcases. 

First,  if $\lambda = \lambda^e \cup (1)$, then  $k=1$ and 
$$\mu=\varphi_2(\lambda) =(\lambda^e_1 - 1, \lambda^e_2 - 1, \dotsc, \lambda_s^e - 1) \cup (1^{\ell_e(\lambda) + 1}) \in \mathcal B_5(n),$$
 as $\mu_r = 1$ and $m_\mu(1)> \ell_{> 1}(\mu)$. Note that in this case $\ell(\mu)$ is odd, and $m_\mu(1) = \ell_{ > 1}(\mu) + j$ with $j \geq 1$ odd. Below, we give a specific example. 

\begin{example}\label{B5c}
Let $n = 27$ and $\lambda = (8, 6, 6, 4, 2, 1)$. Then $\varphi_2(\lambda) = (7, 5, 5, 3, 1^7) \in \mathcal{B}_5(27)$.
\end{example}

Second,  if $\lambda_1^o = \lambda_s^e - 1,$ $\lambda^o \neq (1)$, and $\lambda^o\neq (\lambda_s^e - 1, 1)$, then $k=1$ and 
\begin{align*}
& \mu=\varphi_2(\lambda) =(\lambda^e_1 - 1, \lambda^e_2 - 1, \dotsc, \lambda_s^e - 1) \cup (\lambda_s^e - 1, \lambda^o_2, \dotsc, \lambda^o_s) \cup (1^{\ell_e(\lambda)}) \in \mathcal B_5(n),
\end{align*}

\noindent as $\mu_r = 1$, $m_\mu(1) < \ell_{>1}(\mu),$ $\lambda_s^e - 1$ is repeated, and if $\lambda^o_s>1$, then $m_\mu(1)=\ell_{\geq \mu_{rr}}(\mu)-1$ and if $\lambda^o_s=1$, then $m_\mu(1)=\ell_{\geq \mu_{rr}}(\mu)$.This falls under restriction (I).
Below, we give a specific example.

\begin{example}\label{B5d}
Let $n = 29$ and $\lambda = (8, 6, 6, 5, 3, 1)$. Then $\varphi_2(\lambda) = (7, 5, 5, 5, 3, 1^4) \in \mathcal{B}_5(29)$.
\end{example}

Next, we consider partitions $\lambda$ in Case 5 satisfying $\ell_e(\lambda) \geq 2, \ell_o(\lambda) \geq 1$, and $\lambda^o \neq (\lambda^e_s - 1, 1)$ for which  $k > 1$. The images of these partitions under $\varphi_2$ all fall under restriction (II).

\smallskip

\noindent $\bullet$ First, assume  $3\leq k=\lambda^e_s - \lambda_1^o < \lambda^o_1$. In addition, assume $\lambda_2^o = k$. Then 
\begin{align*}
\mu=\varphi_2(\lambda) &= (\lambda^e_1 - k, \lambda^e_2 - k, \dotsc, \lambda_s^e - k = \lambda_1^o) \cup (\lambda_1^o) \cup (k^{\ell_e(\lambda) + 1}) \cup (\lambda_3^o, \lambda^o_4, \dotsc, \lambda^o_s).
 \end{align*}
We have $\mu\in \mathcal B_5(n)$, as $m_\mu(\mu_r) = \ell_{>\mu_r}(\mu)$ and $m_\mu(a_{> \mu_r}) > 1$. We obtain partitions in $\mathcal B_5(n)$ under restriction (II)(ii). Below, we give a specific example which we illustrate with Ferrers diagrams.  

\begin{example}\label{B5e}
Let $n = 35$ and $\lambda = (10, 8, 8, 5, 3, 1)$. Then $k=3$ and 
$$\varphi_2(\lambda) = (7, 5, 5, 5, 3, 3, 3, 3, 1) \in \mathcal{B}_5(35).$$\smallskip

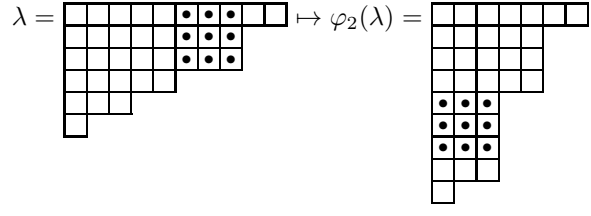
\begin{center}
 \ytableausetup{smalltableaux}
$\lambda =$ \ydiagram [*(white) \bullet]
  {5 + 3,  5+ 3, 5 + 3, 5 + 0, 3 + 0, 1 + 0}
  *[*(white)]{10, 8, 8, 5, 3, 1}\ $\mapsto \varphi_2(\lambda)=$
\ydiagram [*(white) \bullet]
  {7 + 0,  5+ 0, 5 + 0, 5 + 0, 0 + 3, 0 + 3, 0 + 3, 3 + 0, 1 + 0}
  *[*(white)]{7, 5, 5, 5, 3, 3, 3, 3, 1}
\end{center} 
\end{example}
\smallskip

\noindent $\bullet$ If  $3\leq k=\lambda^e_s - \lambda_1^o < \lambda^o_1$ with the additional assumption $\lambda_1^o > k > \lambda_2^o$, then 
\begin{align*}
\mu=\varphi_2(\lambda) &= (\lambda^e_1 - k, \lambda^e_2 - k, \dotsc, \lambda_s^e - k = \lambda_1^o) \cup (\lambda_1^o) \cup (k^{\ell_e(\lambda)}) \cup (\lambda_2^o, \lambda^o_3, \dotsc, \lambda^o_s).
 \end{align*}
We have $\mu\in \mathcal B_5(n)$, as $m_\mu(\mu_r) = \ell_{>\mu_r}(\mu) -1$, $m_\mu(a_{> \mu_r}) > 1,$ so $\mu_{rr} = a_{> \mu_r}$ and $m_\mu(\mu_r) = \ell_{\geq\mu_{rr}}(\mu) - 1$. 
We obtain partitions in $\mathcal B_5(n)$ falling under restriction (II)(i). Below, we give a specific example which we illustrate with Ferrers diagrams.  

\begin{example}\label{B5f}
Let $n = 32$ and $\lambda = (10, 8, 8, 5, 1)$. Then $k=3$ and 
$$\varphi_2(\lambda) = (7, 5, 5, 5, 3, 3, 3, 1) \in \mathcal{B}_5(32).$$
\smallskip

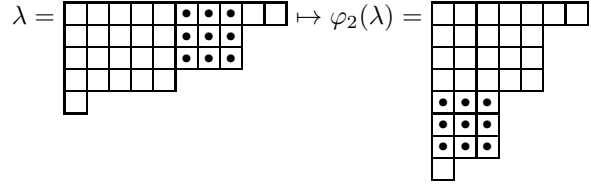
\begin{center}
 \ytableausetup{smalltableaux}
$\lambda =$ \ydiagram [*(white) \bullet]
  {5 + 3,  5+ 3, 5 + 3, 5 + 0, 1 + 0}
  *[*(white)]{10, 8, 8, 5, 1}\ $\mapsto  \varphi_2(\lambda)=$
\ydiagram [*(white) \bullet]
  {7 + 0,  5+ 0, 5 + 0, 5 + 0, 0 + 3, 0 + 3, 0 + 3, 1 + 0}
  *[*(white)]{7, 5, 5, 5, 3, 3, 3, 1}
\end{center} 
\smallskip

\end{example}

\noindent $\bullet$ If  $3\leq k=\lambda^e_s - \lambda_1^o < \lambda^o_1$ with the additional assumption  $\lambda_i^o = k$ for some $i > 2$, then 
\begin{align*}
\mu=\varphi_2(\lambda) =& (\lambda^e_1 - k, \lambda^e_2 - k, \dotsc, \lambda_s^e - k = \lambda_1^o) \cup (\lambda_1^o, \lambda_2^o, \dotsc, \lambda_{i-1}^o) \\
& \cup (k^{\ell_e(\lambda) + 1}) \cup (\lambda_{i+1}^o, \lambda^o_{i+2}, \dotsc, \lambda^o_s).
 \end{align*} 
We have $\mu\in \mathcal B_5(n)$, as $m_\mu(\mu_r) <\ell_{>\mu_r}(\mu)$, $\mu_{rr}$ exists but is not $a_{> \mu_r}$, and $m_\mu(\mu_r) = \ell_{\geq\mu_{rr}}(\mu)$. We obtain partitions in $\mathcal B_5(n)$ falling under condition (II)(i). Below, we give a specific example which we illustrate with Ferrers diagrams. 

\begin{example}\label{B5g}
Let $n = 66$ and $\lambda = (16, 14, 10, 10, 7, 5, 3, 1)$. Then $k=3$ and 
$$\varphi_2(\lambda) = (13, 11, 7, 7, 7, 5, 3, 3, 3, 3, 3, 1) \in \mathcal{B}_5(66).$$\smallskip

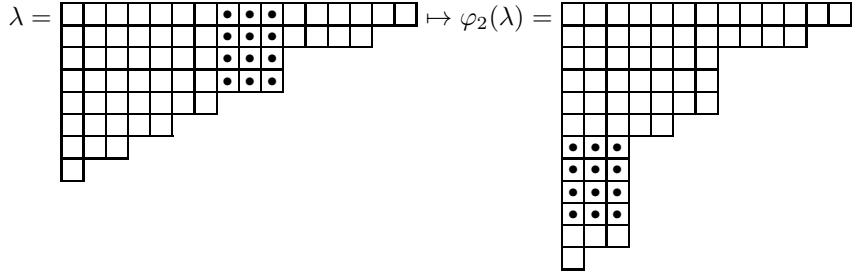
\begin{center}
 \ytableausetup{smalltableaux}
$\lambda =$ \ydiagram [*(white) \bullet]
  {7+ 3, 7+ 3,  7+ 3, 7+ 3, 7 + 0, 5 + 0, 3 + 0, 1 + 0}
  *[*(white)]{16, 14, 10, 10, 7, 5, 3, 1}\ $\mapsto   \varphi_2(\lambda)=$
\ydiagram [*(white) \bullet]
  {13 + 0,  11+ 0, 7 + 0, 7 + 0, 7 + 0, 5 + 0, 0 + 3, 0 + 3, 0 + 3, 0 + 3, 3 + 0, 1 + 0}
  *[*(white)]{13, 11, 7, 7, 7, 5, 3, 3, 3, 3, 3, 1}
\end{center} 

\end{example}\smallskip

\noindent $\bullet$ If  $3\leq k=\lambda^e_s - \lambda_1^o < \lambda^o_1$ with the additional assumption $\lambda_i^o > k > \lambda_{i+1}^o $ for some $i \geq 2$ (recall that if $j>\ell(\lambda)$, we set $\lambda_j=0$ for convenience), then 
\begin{align*}
\varphi_2(\lambda) =& (\lambda^e_1 - k, \lambda^e_2 - k, \dotsc, \lambda_s^e - k = \lambda_1^o) \cup (\lambda_1^o, \lambda_2^o, \dotsc, \lambda_{i}^o) \\
& \cup (k^{\ell(\lambda^e)}) \cup (\lambda_{i+1}^o, \lambda^o_{i+2}, \dotsc, \lambda^o_s).
 \end{align*}
We have $\mu\in \mathcal B_5(n)$, as $m_\mu(\mu_r) <\ell_{>\mu_r}(\mu)$, $\mu_{rr}$ exists but is not $a_{> \mu_r}$, and $m_\mu(\mu_r) = \ell_{\geq\mu_{rr}}(\mu)-1$. We obtain partitions in $\mathcal B_5(n)$ falling under restriction (IIi) \ccom{(II)(i)}. Below, we give a specific example which we illustrate with Ferrers diagrams. \smallskip

\begin{example}\label{B5h}
Let $n = 63$ and $\lambda = (16, 14, 10, 10, 7, 5, 1)$. Then $k=3$ and  
$$\varphi_2(\lambda) = (13, 11, 7, 7, 7, 5, 3, 3, 3, 3, 1) \in \mathcal{B}_5(63).$$\smallskip

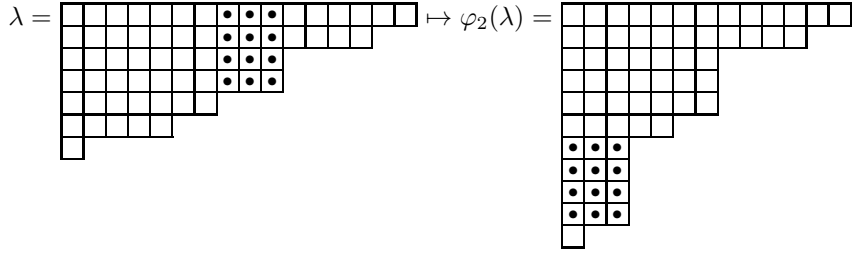
\begin{center}
 \ytableausetup{smalltableaux}
$\lambda =$ \ydiagram [*(white) \bullet]
  {7+ 3, 7+ 3,  7+ 3, 7+ 3, 7 + 0, 5 + 0, 1 + 0}
  *[*(white)]{16, 14, 10, 10, 7, 5, 1}\ $\mapsto \varphi_2(\lambda)=$
\ydiagram [*(white) \bullet]
  {13 + 0,  11+ 0, 7 + 0, 7 + 0, 7 + 0, 5 + 0, 0 + 3, 0 + 3, 0 + 3, 0 + 3, 1 + 0}
  *[*(white)]{13, 11, 7, 7, 7, 5, 3, 3, 3, 3, 1}
\end{center} 

\end{example}\smallskip

Lastly, we consider  partitions $\lambda$ in Case 5, satisfying $\ell_e(\lambda) \geq 2, \ell_o(\lambda) \geq 1$, and $\lambda^o \neq (\lambda^e_s - 1, 1),$ and such that $\lambda^e_s - \lambda_1^o \geq \lambda^o_1=k\geq 3$. 
\smallskip

\noindent $\bullet$ If $\lambda^e_s - \lambda_1^o > \lambda^o_1$, then 
\begin{align*}
\mu=\varphi_2(\lambda) = &(\lambda^e_1 - k, \lambda^e_2 - k, \dotsc, \lambda_s^e - k \neq \lambda_1^o) \cup (\lambda_1^o = k) \\
& \cup (k^{\ell(\lambda^e) }) \cup (\lambda_2^o, \lambda^o_3, \dotsc, \lambda^o_s).
 \end{align*}
We have $\mu\in \mathcal B_5(n)$, as $m_\mu(\mu_r) = \ell_{>\mu_r}(\mu) + 1$ and $\ell_{ \geq \mu_r}(\mu)$ is odd. We obtain partitions in $\mathcal B_5(n)$ falling under restriction (II)(iii). Below, we give a specific example which we illustrate with Ferrers diagrams. \smallskip 

\begin{example}\label{B5i}
Let $n = 54$ and $\lambda = (16, 14, 10, 10, 3, 1)$. Then $k = 3$ and 
$$\varphi_2(\lambda) = (13, 11, 7, 7, 3, 3, 3, 3, 3, 1) \in \mathcal{B}_5(63).$$\smallskip

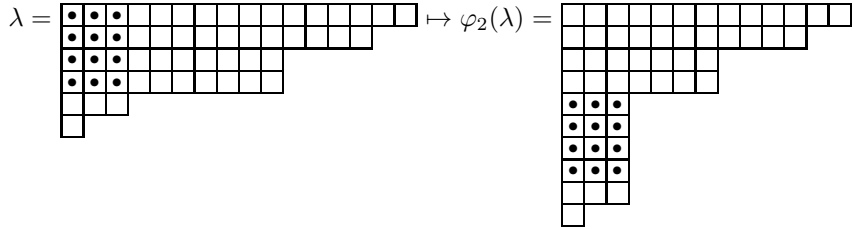
\begin{center}
 \ytableausetup{smalltableaux}
$\lambda =$ \ydiagram [*(white) \bullet]
  {0 + 3, 0+ 3,  0+ 3, 0+ 3, 3 + 0, 1 + 0}
  *[*(white)]{16, 14, 10, 10, 3, 1}\ $\mapsto \varphi_2(\lambda)=$
\ydiagram [*(white) \bullet]
  {13 + 0,  11+ 0, 7 + 0, 7 + 0, 0 + 3, 0 + 3, 0 + 3, 0 + 3, 3 + 0, 1 + 0}
  *[*(white)]{13, 11, 7, 7, 3, 3, 3, 3, 3, 1}
\end{center}
\end{example}

\noindent $\bullet$ If $\lambda^e_s - \lambda_1^o = \lambda^o_1$, then 

\begin{align*}
\mu=\varphi_2(\lambda) =& (\lambda^e_1 - k, \lambda^e_2 - k, \dotsc, \lambda_s^e - k = \lambda_1^o) \cup (\lambda_1^o = k) \\
& \cup (k^{\ell(\lambda^e) }) \cup (\lambda_2^o, \lambda^o_3, \dotsc, \lambda^o_s).
 \end{align*}
We have $\mu\in \mathcal B_5(n)$, as $m_\mu(\mu_r) = \ell_{\geq \mu_r}(\mu) + j$ with $j > 1$ odd and $\ell_{\geq \mu_r}(\mu)$ is odd. We obtain partitions in $\mathcal B_5(n)$ falling under restriction (II)(iii). Below, we give a specific example which we illustrate with Ferrers diagrams. \smallskip

\begin{example}\label{B5h1}
Let $n = 59$ and $\lambda = (16, 14, 10, 10, 5, 3, 1)$. Then $k = 5$ and

$$\varphi_2(\lambda) = (11, 9, 5, 5, 5, 5, 5, 5, 5, 3, 1) \in \mathcal{B}_5(59).$$\\

\begin{center}
 \ytableausetup{smalltableaux}
$\lambda =$ \ydiagram [*(white) \bullet]
  {0+ 5, 0 + 5,  0+ 5, 0+ 5, 5+ 0, 3 + 0, 1 + 0}
  *[*(white)]{16, 14, 10, 10, 5, 3, 1}\ $\mapsto \varphi_2(\lambda)=$
\ydiagram [*(white) \bullet]
  {11 + 0,  9+ 0, 5 + 0, 5 + 0, 0 + 5, 0 + 5, 0 + 5, 0 + 5, 5 + 0, 3 + 0, 1 + 0}
  *[*(white)]{11, 9, 5, 5, 5, 5, 5, 5, 5, 3, 1}
\end{center}
\end{example}

\smallskip

\subsection{The inequality $p_{\eu}^{\od}(n) <p^{\ed}_{\ou}(n)$}

 Let $n\geq 0$. We define a bijection $$\varphi_3: \mathcal P_{\eu}^{\od}(n)\to \mathcal P_{\ou}^{\ed}(n).$$

 Let $\lambda\in \mathcal P_{\eu}^{\od}(n)$. We consider three cases.

\noindent \underline{Case 1:} If $\lambda^e=\emptyset$, define $\varphi_3(\lambda):=\lambda$. \smallskip
 
 \noindent \underline{Case 2:} If $\lambda^o=\emptyset$ and  $\lambda\in \mathcal D(n)$, define $\varphi_3(\lambda):=\lambda$. \smallskip
 
 \noindent \underline{Case 3:} 
If $\lambda^o=\emptyset$ and  $\lambda\not \in \mathcal D(n)$ or $\lambda^o, \lambda^e\neq \emptyset$, define $$\varphi_3(\lambda):=(\lambda_1-1,  \lambda_2-1, \ldots, \lambda_{\ell(\lambda)}-1)\cup (1^{\ell(\lambda)}),$$
  the partition obtained from $\lambda$ by subtracting one from each of its  parts  and inserting $\ell(\lambda)$ parts equal to $1$. 

Let $\widetilde{\mathcal P}_{\ou}^{\ed}(n)$ be the subset of partitions $\mu \in \mathcal P_{\ou}^{\ed}(n)$ satisfying one of the following conditions.
\begin{itemize}
    \item[(i)] all parts of $\mu$ are odd and distinct; 
    \item[(ii)] all parts of $\mu$ are even and distinct; 
    \item[(iii)] $\mu^o\neq \emptyset$, $\ell(\mu)\geq 4$, $\ell(\mu)$ even, $m_\mu(1)\geq 2$, $m_\mu(1)\geq \ell_{>1}(\mu)$, and 
 if $\mu^e=\emptyset$ and $1$ is the only repeated part, then either $m_\mu(1)=\ell(\mu)$ or $m_\mu(1)\geq \ell_{>1}(\mu)+4$.
\end{itemize}
Then, $\varphi_3: \mathcal P_{\eu}^{\od}(n)\to \widetilde{\mathcal P}_{\ou}^{\ed}(n)$ is a bijection. To see that $\varphi_3$ is invertible, let $\mu\in \widetilde{\mathcal P}_{\ou}^{\ed}(n)$. If $\mu=\mu^o$
 or $\mu=\mu^e\in \mathcal D(n)$, then $\varphi_3^{-1}(\mu)=\mu$. Otherwise, $\mu$ satisfies all conditions in (iii) above and $\varphi_3^{-1}(\mu)$ is the partition obtained from $\mu$ by removing $\ell(\mu)/2$ parts equal to $1$
and adding one to each of the first $\ell(\mu)/2$ parts of $\mu$.

\begin{example}\label{phi3}
Let $n = 22$ and $\lambda = (10, 8, 2, 2)$. Then 

$$\varphi_3(\lambda) = (9, 7, 1^6) \in \widetilde{\mathcal P}_{\ou}^{\ed}(22).$$\smallskip

\begin{center}
 \ytableausetup{smalltableaux}
$\lambda =$ \ydiagram [*(white) \bullet]
  {9+1, 7 + 1,  1+ 1, 1 + 1}
  *[*(white)]{10, 8, 2, 2}\ $\mapsto \varphi_3(\lambda)=$
\ydiagram [*(white) \bullet]
  {9 + 0, 7 + 0, 1 + 0, 1 + 0, 1, 1, 1, 1}
  *[*(white)]{9, 7, 1, 1, 1, 1}
\end{center}\smallskip

To reverse, start with $\mu = (9, 7, 1^6)$. Then $\ell(\mu) = 8,$ so remove $4$ parts of size $1$ and add one to the first 4 parts. This returns $\varphi_3^{-1}(\mu) = (9 + 1, 7 + 1, 1 + 1, 1+ 1) = (10, 8, 2, 2) \in \mathcal P_{\eu}^{\od}(22).$
\end{example}

Clearly, $p_{\eu}^{\od}(1) =p^{\ed}_{\ou}(1)=1$. If $n=2$, the partition $(1,1)\in \mathcal P_{\ou}^{\ed}(2)$ is not in $\widetilde{\mathcal P}_{\ou}^{\ed}(2)$. If $n\geq 3$ is odd, then the partition $(n-1, 1)\in {\mathcal P}_{\ou}^{\ed}(n)$ is not in $\widetilde{\mathcal P}_{\ou}^{\ed}(n)$. If $n\geq 4$ is even, then the partition $(n-2,1, 1)\in {\mathcal P}_{\ou}^{\ed}(n)$ is not in $\widetilde{\mathcal P}_{\ou}^{\ed}(n)$.  Together with the combinatorial injection $\varphi_3$, this proves the following theorem. 
\begin{theorem}
    The inequality $p_{\eu}^{\od}(n) \leq p^{\ed}_{\ou}(n)$ holds for all $n\geq 0$. Moreover, $p_{\eu}^{\od}(n) <p^{\ed}_{\ou}(n)$ for $n\geq 2$.
\end{theorem}

\smallskip

\subsection{The inequality $p_{\od}^{\ed}(n) < p_{\eu}^{\od}(n)$}
One can verify directly that $p_{\od}^{\ed}(7)=4$ and $p_{\eu}^{\od}(7)=3$.

Let $n\geq 8$. We create an injection $$\varphi_4: \mathcal P_{\od}^{\ed}(n)\to \mathcal P_{\eu}^{\od}(n).$$

Let $\lambda\in \mathcal P_{\od}^{\ed}(n)$. We consider five cases.  For $1\leq i\leq 5$, we denote by $\mathcal C_i(n)$ the image  of the set of all partitions in Case $i$ under $\varphi_4$.  At the end of this section, we give examples for each case.\smallskip

\noindent \underline{Case 1:} If $\lambda^e=\emptyset$ or $\lambda^o=\emptyset$, define $\varphi_4(\lambda):=\lambda$.  Then, $$\mathcal C_1(n)=\{\mu\in \mathcal P^{\od}_{\eu}(n)\ \big| \ \mu=\mu^o \in \mathcal D(n) \text{ or } \mu=\mu^e \in \mathcal D(n)\}.$$
\smallskip 

\noindent \underline{Case 2:} If $\lambda^o, \lambda^e\neq \emptyset$ and $1\not \in \lambda$, define $\varphi_4(\lambda)$ to be the partition  obtained from $\lambda$ by subtracting one from each of the last $\ell_m(\lambda)$ parts and adding one to each of the first $\ell_m(\lambda)$ parts. Note that in this case, $\varphi_4$ is similar to the mapping $\varphi_1$ defined in section \ref{s1}. Then, 
$$\mathcal C_2(n)=\{\mu\in \mathcal P^{\od}_{\eu}(n)\ \big| \  \mu\in \mathcal D(n), \ \mu^o, \mu^e\neq \emptyset, \ \mu^o_s-\mu^e_1\geq 3\}.$$
\smallskip

\noindent \underline{Case 3:} Suppose $\lambda^o, \lambda^e\neq \emptyset$, $1 \in \lambda$, $\ell_o(\lambda)\geq 2$, and if $\ell_o(\lambda)=2$,  then $\lambda^e_s-\lambda^o_1\neq 1$. Let $\eta$ be the partition obtained from $\lambda$ by removing the first part and the last two parts. Then $\eta$ satisfies the conditions of Case 1 or Case 2 and $\varphi_4(\eta)$ is defined. If $\lambda_{\ell(\lambda)-1}=2k+1$ for some $k\geq 1$, then $\lambda_1\geq 2k+2$ and we write  $\lambda_1=q(2k+2)+r$ with $q\geq 1$, $r$ even and $0\leq r<2k+2$. We define  $$\varphi_4(\lambda):=\begin{cases}\varphi_4(\eta)\cup (2k+2)^{q+1}, & \text{ if } r=0, \\ \varphi_4(\eta)\cup (2k+2)^{q+1}\cup(r), & \text{ if } r>0. \end{cases}$$ Then, $\mathcal C_3(n)$ is the subset of partitions $\mu\in \mathcal P^{\od}_{\eu}(n)$ satisfying all of the following conditions:
\begin{itemize}
\item $\ell(\mu)\geq 3$;
\item $\mu$ has a single repeated part and the repeated part is even, of size at least $4$, and it is either the smallest or second smallest part of $\mu$;
    \item the partition $\widetilde \mu$ consisting of the parts of   $\mu$ that are larger than the repeated part is in $\mathcal C_1\cup\mathcal C_2$; 
    \item if $\widetilde \mu \neq \emptyset$ and  the repeated part of $\mu$ is $2k+2$, then $|\mu|-|\widetilde \mu|-(2k+2)>\mu_1$. 
\end{itemize} 
\smallskip

\noindent \underline{Case 4:} ($n$ even) Suppose $\lambda^o, \lambda^e\neq \emptyset$, $1 \in \lambda$,  $\ell_o(\lambda)=2$, and $\lambda^e_s-\lambda^o_1= 1$. If $\lambda^o_1=2k+1$ for some $k\geq 1$, then $$\varphi_4(2k+2, 2k+1, 1):=(2k+2, 2k+2),$$ and if $\ell(\lambda)\geq 4$, then $$\varphi_4(\lambda):=(\lambda^e_2, \ldots, \lambda^e_{\ell_e(\lambda)-1}, \lambda^e_s=2k+2, 2k+2)\cup(2^{\lambda_1/2}),$$ the partition obtained from $\lambda$ by removing the first part and the last two parts of $\lambda$ and inserting a part equal to $2k+2$ and $\lambda_1/2$ parts equal to $2$. Note that $\lambda_1\geq 6$.

Then, $\mathcal C_4(n)$ is the subset of partitions $\mu\in \mathcal P^{\od}_{\eu}(n)$ satisfying all of the following conditions:
\begin{itemize}
    \item $\mu=\mu^e\not\in \mathcal D(n)$;
    \item If $\ell(\mu)>2$, then $2\in \mu$, the two smallest part sizes in $\mu$ are repeated and the repeated part not equal to $2$ has multiplicity exactly two. Moreover, $2m_\mu(2)>\mu_1$.
\end{itemize}\smallskip

\noindent \underline{Case 5:} ($n$ odd, $n\geq 9$). Suppose $\lambda^e\neq \emptyset$, $\lambda^o=(1)$. We have several subcases. 
\begin{itemize}
    \item[(i)] If $\ell(\lambda)=2$, then $\varphi_4(\lambda)=\varphi_4(n-1,1):=(n-4, 2, 2)$. Then $$\mathcal C_{5,(i)}(n)=\{\mu\in \mathcal P^{\od}_{\eu}(n)\ \big| \  \ell(\mu)=3, \ \ell_0(\mu)=1, \ m_\mu(2)=2\}.$$
    \item[(ii)] If $\ell(\lambda)=3$, then $\varphi_4(\lambda):=\varphi_4(\lambda_1, \lambda_2, 1)=(\lambda_2+1, 2^{\lambda_1/2})$.
    Then, since $n\geq 9$,   
   $$\mathcal C_{5,(ii)}(n)=\{\mu\in \mathcal P^{\od}_{\eu}(n)\ \big| \  \ell_o(\mu)=1, \ \mu=(\mu_1, 2^{m_\mu(2)}), \ m_\mu(2)\geq 3, \ 2m_\mu(2)>\mu_1\}.$$ 
    \item[(iii)] If $\ell(\lambda)=4$ and $2\in \lambda$, then $\varphi_4(\lambda):=\varphi_4(\lambda_1, \lambda_2, 2,1)=(\lambda_1+1,  2^{\lambda_2/2+1})$.  Then,   $$\mathcal C_{5,(iii)}(n)=\{\mu\in \mathcal P^{\od}_{\eu}(n)\ \big| \ \ell_o(\mu)=1, \  \mu=(\mu_1^o, 2^{m_\mu(2)}), \ m_\mu(2)\geq 3, \ 2m_\mu(2)<\mu_1\}.$$
    \item[(iv)] If $\ell(\lambda)\geq 4$ and $2\not\in \lambda$, then $\varphi_4(\lambda):=(\lambda_1+1, \lambda_2, \ldots, \lambda_{\ell(\lambda)-2}, 2^{\lambda_{\ell(\lambda)-1}/2})$, the partition obtained from $\lambda$ by adding one to the first part of $\lambda$, removing the last two parts of $\lambda$,  and inserting $\lambda_{\ell(\lambda)-1}/2$ parts equal to $2$. Then, $\mathcal C_{5,(iv)}(n)$ is the subset of partitions $\mu\in \mathcal P^{\od}_{\eu}(n)$ 
    satisfying all of the following conditions:
\begin{itemize}
    \item[$\bullet$] $\ell_o(\mu)=1$;
    \item[$\bullet$] $\mu^e$ has at least one part greater than $2$;
    \item[$\bullet$] the only repeated part of $\mu$ is $2$ and $2m_\mu(2)<a_{>2}$. 
\end{itemize}
Note that if $\mu\in \mathcal C_{5,(iv)}(n)$, then $2m_\mu(2)<
\mu_1$. 
\item[(v)] If $\ell(\lambda)\geq 5$ and $2\in \lambda$, then 
$\varphi_4(\lambda):=(\lambda_2+1, \lambda_3, \ldots, \lambda_{\ell(\lambda)-2}, 2^{\lambda_1/2+1}),$ the partition obtained from $\lambda$ by 
removing the first and last part of $\lambda$, adding one to second part of $\lambda$, and inserting $\lambda_1/2$ parts equal to $2$.  Then, $\mathcal C_{5,(v)}(n)$ is the subset of partitions $\mu\in \mathcal P^{\od}_{\eu}(n)$ satisfying all of the following conditions:
\begin{itemize}
    \item[$\bullet$] $\ell_o(\mu)=1$;
    \item[$\bullet$] $\mu^e$ has at least two parts greater than $2$;
    \item[$\bullet$] the only repeated part of $\mu$ is $2$ and $2m_\mu(2)>\mu_1+2$. 
\end{itemize}
\end{itemize}
Let $\mathcal C_5(n)=\mathcal C_{5,(i)}\cup \mathcal C_{5,(ii)}\cup \mathcal C_{5,(iii)}\cup \mathcal C_{5,(iv)}\cup \mathcal C_{5,(v)}$. Then, $\mathcal C_{5}(n)$ is the subset of partitions $\mu\in \mathcal P^{\od}_{\eu}(n)$ satisfying all of the following conditions: 
\begin{itemize}
    \item $\ell_o(\mu)=1$;
    \item the only repeated part of $\mu$ is $2$;

    \item if $\mu$ has exactly one even part greater than $2$,  then $2m_\mu(2)<a_{>2}$;
    
   \item if $\mu$ has at least two  even parts greater than $2$, then $2m_\mu(2)<a_{>2}$ or $2m_\mu(2)>\mu_1+2$.
\end{itemize}

Clearly, the sets $\mathcal C_i(n)$, $1\leq i\leq 5$, are mutually disjoint. \smallskip

The mapping  $$\varphi_4:\mathcal P_{\od}^{\ed}(n) \to \displaystyle \bigcup_{i=1}^5\mathcal C_i(n)$$ is a bijection. To see that $\varphi_2$ is invertible, let $\mu\in \ds\bigcup_{i=1}^5\mathcal C_i(n)$. \smallskip

If $\mu\in \mathcal C_1(n)$, then $\varphi_4^{-1}(\mu)=\mu$. 

If $\mu\in \mathcal C_2(n)$, then $\varphi_4^{-1}(\mu)$ is the partition obtained from $\mu$ by subtracting one from the first $\ell_m(\mu)$ parts of $\mu$ and adding one to the last $\ell_m(\mu)$ parts of $\mu$. 

If $\mu\in \mathcal C_3(n)$, let $\widetilde \mu$ be as in the description of $\mathcal C_3(n)$. Then $$\varphi_4^{-1}(\mu)=\varphi_4^{-1}(\widetilde \mu) \cup(|\mu|-|\widetilde \mu|-(2k+2), 2k+1, 1).$$

If $\mu\in \mathcal C_4(n)$, then
\begin{itemize}
    \item if $\ell(\mu)=2$, then $\varphi_4^{-1}(\mu)=(\mu_1, \mu_2-1, 1)$;
    \item if $\ell(\mu)>2$ and the repeated part greater than $2$ is $2k+2$ for some $k\geq 1$, then $$\varphi_4^{-1}(\mu)=(\mu\setminus (2k+2, 2^{m_\mu(2)}))\cup (2m_\mu(2), 2k+1, 1).$$
\end{itemize}

If $\mu\in \mathcal C_5(n)$, then
\begin{itemize}
    \item if $\ell(\mu)=3$, then $\varphi_4^{-1}(\mu)=\varphi_4^{-1}(\mu_1,2,2)=(\mu_1+3, 1)$;
    \item If $\ell(\mu)\geq 4$ and $\mu$ has no even part greater than $2$, then $$\varphi_4^{-1}(\mu)=\begin{cases}
        (2m_\mu(2), \mu_1-1, 1), & \text{ if } 2m_\mu(2)>\mu_1, \\ (\mu_1-1, 2m_\mu(2)-2, 2, 1), & \text{ if } 2m_\mu(2)<\mu_1;
    \end{cases} $$
    \item If $\ell(\mu)\geq 4$ and $\mu$ has an even part greater than $2$, then $$\varphi_4^{-1}(\mu)=\begin{cases}(\mu\setminus(\mu_1, 2^{m_\mu(2)}))\cup(\mu_1-1, 2m_\mu(2), 1), & \text{ if } 2m_\mu(2)<a_{>2}, \\ (\mu\setminus(\mu_1, 2^{m_\mu(2)-1}))\cup(2m_\mu(2)-2,\mu_1-1,  1), & \text{ if } 2m_\mu(2)>\mu_1+2.
    \end{cases}$$
\end{itemize}

If $n\geq 8$ is even, then the partition $(n-4,2,2)\in \mathcal P^{\od}_{\eu}(n)$ is not in $\bigcup_{i=1}^5 \mathcal C_i(n)$. If $n\geq 9$ is odd, then for $n=4k+1$ for some $k\geq 2$, the partition $(2k+1, 2k)\in \mathcal P^{\od}_{\eu}(n)$ is not in $\bigcup_{i=1}^5 \mathcal C_i(n)$, and for  $n=4k+3$ for some $k\geq 2$, the partition $(2k+1, 2k,2)\in \mathcal P^{\od}_{\eu}(n)$ is not in $\bigcup_{i=1}^5 \mathcal C_i(n)$.  Together with the combinatorial injection $\varphi_4$, this  proves the following theorem.

\begin{theorem} The inequality $p_{\od}^{\ed}(n) < p_{\eu}^{\od}(n)$ holds for all $n\geq 8$.
    \end{theorem}

\subsubsection{Examples}  To help clarify the map $\varphi_4$, we provide examples for the many cases. 

\begin{example}
Let $n = 35$ and $\lambda = (16, 12, 7)$. This falls under Case 2 with $\ell_m(\lambda) = \ell_o(\lambda)=1$. Thus, $\varphi_4(\lambda) = (17, 12, 6)$.\medskip

\small
\begin{center}
 \ytableausetup{smalltableaux}
$\lambda =$ \ydiagram [*(white) \bullet]
  {16 +0, 12 +0, 6 + 1}
  *[*(white)]{16, 12, 7}\ $\mapsto \  \varphi_4(\lambda) =$
\ydiagram [*(white) \bullet]
  {16 +1, 12 + 0, 6 + 0}
  *[*(white)]{17, 12, 6}
\end{center} 
\end{example}\smallskip

\begin{example}
Let $n = 61$ and $\lambda = (20, 16, 12, 7, 5, 1)$. This falls under Case 3 with $k = 2, q = 3, r = 2$, and $\eta = (16, 12, 7)$. Thus, $\varphi_4(\eta) = (17, 12, 6)$ and $\varphi_4(\lambda) = (17, 12, 6^5, 2)$.\medskip

\small
\begin{center}
 \ytableausetup{smalltableaux}
$\lambda =$ \ydiagram [*(white) \bullet]
  {20 , 16 + 0, 12+ 0, 7 + 0, 5, 1}
  *[*(white) \star ]{20 + 0 , 16 + 0, 12+ 0, 6 + 1, 5 + 0, 1 + 0}
  *[*(white)]{20, 16, 12, 7, 5, 1}\ $\mapsto \  \varphi_4(\lambda) =$
\ydiagram [*(white) \bullet]
  {17 + 0, 12 + 0, 6 + 0, 6, 6, 6, 6,2}
  *[*(white) \star ]{16 +1, 12 + 0, 6 + 0, 6 + 0, 6 + 0, 6 + 0, 6 + 0, 2+0}
  *[*(white)]{17, 12, 6, 6, 6, 6, 6,2}
\end{center} 
\end{example}\smallskip

\begin{example}
Let $n = 60$ and $\lambda = (20, 16, 12, 6, 5, 1)$. This falls under Case 4 with $k = 2$. Thus, $\varphi_4(\lambda) = (16, 12, 6^2, 2^{10})$.\medskip

\small
\begin{center}
 \ytableausetup{smalltableaux}
$\lambda =$ \ydiagram [*(white) \bullet]
  {20 , 16 + 0, 12+ 0, 6 + 0, 5 + 0, 1 + 0}
  *[*(white) \star ]{20 + 0 , 16 + 0, 12+ 0, 6 + 0, 5, 1}
  *[*(white)]{20, 16, 12, 6, 5, 1}\ $\mapsto \  \varphi_4(\lambda) =$
\ydiagram [*(white) \bullet]
  {16 + 0, 12+0, 6+0, 6 + 0, 2, 2, 2, 2, 2, 2, 2, 2, 2, 2}
   *[*(white) \star ]{16 + 0, 12+0, 6+0, 6, 2 + 0, 2 + 0, 2 + 0, 2 + 0, 2 + 0, 2 + 0, 2 + 0, 2 + 0, 2 + 0, 2 + 0}
  *[*(white)]{16, 12, 6, 6, 2, 2, 2, 2, 2, 2, 2, 2, 2, 2}
\end{center} 
\end{example}\smallskip

The remaining examples belong to Case 5: $\lambda^e \neq \emptyset$ and $\lambda^o = (1)$, which is broken into several subcases. 

\begin{example}
Let $n = 9$ and $\lambda = (8, 1)$. This falls under Subcase $5(i)$. Thus, $\mu=\varphi_4(\lambda) = (n - 4, 2^2) = (5, 2^2)$. Note that $m_\mu(2) = 2$.
\end{example}

\begin{example}
Let $n = 13$ and $\lambda = (8, 4, 1)$. This falls under Subcase $5(ii)$. Thus, $1$ is added to the second part, and the first part is split into $4$ copies of $2$ to give $\mu=\varphi_4(\lambda) = (5, 2^4)$. Note that $2m_\mu(2) > \mu_1$.\medskip

\begin{center}
 \ytableausetup{smalltableaux}
$\lambda =$ \ydiagram [*(white) \bullet]
  {8,  4 + 0, 1 + 0}
  *[*(white) \star ]{8 + 0,  4 + 0, 1}
  *[*(white)]{8, 4, 1}\ $\mapsto \  \varphi_4(\lambda) =$
\ydiagram [*(white) \bullet]
  {5 + 0, 2, 2, 2, 2}
  *[*(white) \star ]{4 + 1, 2 + 0, 2 + 0, 2 + 0, 2 + 0}
  *[*(white)]{5, 2, 2, 2, 2}
\end{center}

\end{example}\smallskip

\begin{example}
Let $n = 25$ and $\lambda = (12, 10, 2, 1)$. This falls under Subcase $5(iii)$. Thus, $1$ is added to the first part, and the second part is split into $5$ copies of $2$ to give $\mu=\varphi_4(\lambda) = (13, 2^6)$. Note that $2m_\mu(2) < \mu_1$. \medskip

\begin{center}
 \ytableausetup{smalltableaux}
$\lambda =$ \ydiagram [*(white) \bullet]
  {12 + 0, 10, 2 + 0, 1 + 0}
  *[*(white) \star ]{12 + 0, 10 + 0, 2 + 0, 1}
  *[*(white)]{12, 10, 2, 1}\ $\mapsto \  \varphi_4(\lambda) =$
\ydiagram [*(white) \bullet]
  {13 + 0, 2 + 0, 2, 2, 2, 2, 2}
  *[*(white) \star ]{12 + 1, 2 + 0, 2, 2, 2, 2, 2}
  *[*(white)]{13, 2, 2, 2, 2, 2, 2}
\end{center} 

\end{example}\smallskip

\begin{example}
Let $n = 31$ and $\lambda = (12, 10, 8, 1)$. This falls under Subcase $5(iv)$. Thus, $1$ is added to the first part, and the second to last part is split into $4$ copies of $2$ to give $\mu=\varphi_4(\lambda) = (13, 10, 2^4)$. Note that $2m_\mu(2) < a_{>2} = 10$.\medskip

\begin{center}
 \ytableausetup{smalltableaux}
$\lambda =$ \ydiagram [*(white) \bullet]
  {12+ 0, 10 + 0, 8,  1 + 0}
   *[*(white) \star ]{12+ 0, 10 + 0, 8,  1}
  *[*(white)]{12, 10, 8, 1}\ $\mapsto \  \varphi_4(\lambda) =$
\ydiagram [*(white) \bullet]
  {13 + 0, 10+ 0, 2, 2, 2, 2}
  *[*(white) \star ]{12 + 1, 10+ 0, 2, 2, 2, 2}
  *[*(white)]{13, 10, 2, 2, 2, 2}
\end{center} 

\end{example}\smallskip

\begin{example}
Let $n = 33$ and $\lambda = (12, 10, 8, 2, 1)$. This falls under Subcase $5(v)$. Thus, $1$ is added to the second part, and the first part is split into $6$ copies of $2$ to give $\mu=\varphi_4(\lambda) = (11, 8, 2^7)$. Note that $2m_\mu(2) > \mu_1 + 2$.\medskip

\begin{center}
 \ytableausetup{smalltableaux}
$\lambda =$ \ydiagram [*(white) \bullet]
  {12, 10 + 0, 8 + 0, 2 + 0, 1 + 0}
  *[*(white) \star ]{12, 10 + 0, 8 + 0, 2 + 0, 1}
  *[*(white)]{12, 10, 8, 2, 1}\ $\mapsto \  \varphi_4(\lambda) =$
\ydiagram [*(white) \bullet]
  {11 + 0, 8 + 0, 2 + 0, 2, 2, 2, 2, 2, 2}
  *[*(white) \star ]{10 + 1, 8 + 0, 2 + 0, 2, 2, 2, 2, 2, 2}
  *[*(white)]{11, 8, 2, 2, 2, 2, 2, 2, 2}
\end{center} 

\end{example}\smallskip

    \section{proof of Conjecture \ref{conj1}}\label{sec_conj}

    Recall that $\overline{\mathcal P}_{\eu}^{\ou}(n)$ denotes the subset of partitions in $\mathcal P_{\eu}^{\ou}(n)$ in which the largest even part appears with odd multiplicity and all other parts appear with even multiplicity. A partition of $n$  with only odd parts is considered to be in  $\overline{\mathcal P}_{\eu}^{\ou}(n)$ with a single part equal to $0$ (and multiplicity one). We  set $\overline p_{\eu}^{\ou}(n):=|\overline{\mathcal P}_{\eu}^{\ou}(n)|$. 
Similarly,  $\overline{\mathcal P}^{\eu}_{\ou}(n)$ denotes  the subset of partitions in $\mathcal P^{\eu}_{\ou}(n)$ in which both even and odd parts occur, the largest even part and the largest odd part each appear with odd multiplicity, and all other parts appear with even multiplicity. We  set $\overline p^{\eu}_{\ou}(n):=|\overline{\mathcal P}^{\eu}_{\ou}(n)|$. Clearly, if $n$ is odd, $\overline p_{\eu}^{\ou}(n)=0$ and if $n$ is even, $\overline p^{\eu}_{\ou}(n)=0$. Moreover, if $n\geq 1$, then all partitions in  $\overline{\mathcal P}^{\eu}_{\ou}(2n+1)$ have even length.   

We prove the following theorem. 
\begin{theorem}
Let $n\geq 3$. Then,  $$\overline p_{\eu}^{\ou}(2n)<\overline p^{\eu}_{\ou}(2n+1).$$ 
\end{theorem}
\begin{proof}
 Let $n\geq 3$. We create an injection $$\varphi_5:\overline{\mathcal P}_{\eu}^{\ou}(2n)\to \overline{\mathcal P}^{\eu}_{\ou}(2n+1).$$ 
 Let $\lambda\in \overline{\mathcal P}_{\eu}^{\ou}(2n)$.  We consider three cases.  For $1\leq i\leq 3$, we denote by $\mathcal E_i(n)$ the image  of the set of all partitions in Case $i$ under $\varphi_5$.  At the end of this section, we give examples for each case.\smallskip

\noindent \underline{Case 1:}
If $\lambda^e=\emptyset$, define $$\varphi_5(\lambda)=\varphi_5(\lambda_1, \lambda_2, \ldots, \lambda_{\ell(\lambda)}):=(\lambda_1+1, \lambda_2, \ldots, \lambda_{\ell(\lambda)}).$$ Then $$\mathcal E_1(2n+1)=\{\mu\in \overline{\mathcal P}^{\eu}_{\ou}(2n+1)\ \big| \ \ell_e(\mu)=1, \mu_1-\mu_2=1\}.$$
\smallskip 

 \noindent \underline{Case 2:}   
If $\lambda^o=\emptyset$, define $$\varphi_5(\lambda):=\lambda\cup(1).$$ Then $$\mathcal E_2(2n+1)=\{\mu\in \overline{\mathcal P}^{\eu}_{\ou}(2n+1)\ \big| \ \mu^o=(1)\}.$$
\smallskip 

Note that if  $\mathcal E_1(2n+1)\cap \mathcal E_2(2n+1)\neq\emptyset$, then $n=1$. \smallskip

 \noindent \underline{Case 3:}   
If $\lambda^e, \lambda^o\neq\emptyset$, define $$\varphi_5(\lambda)=\varphi_5(\lambda_1, \lambda_2, \ldots, \lambda_{\ell(\lambda)}):=(\lambda_1+\lambda_2, \lambda_3-1, \ldots, \lambda_{\ell(\lambda)}-1)\cup (1^{{\ell(\lambda)}-1}),$$ the partition obtained from $\lambda$ by merging the first two parts of $\lambda$, subtracting one from each of the remaining parts, and inserting $\ell(\lambda) - 1$ parts of size $1$. Then 
$$\mathcal E_3(2n+1)=\left\{\mu\in \overline{\mathcal P}^{\eu}_{\ou}(2n+1)\ \left| \begin{array}{l} \mu_1\equiv 2\text{ mod }4,  \ell_o(\mu)\geq 3, \\  \ \\  \mu_1-\mu_2\geq \mu_2+2\geq 3,  m_\mu(1)\geq \ell_{>1}(\mu)\end{array}\right.\right\}.$$
\smallskip 

Clearly, the sets $\mathcal E_i(2n+1)$, $1\leq i\leq 3$, are mutually disjoint. \smallskip

The mapping  $$\varphi_5:\overline{\mathcal P}_{\eu}^{\ou}(2n) \to \displaystyle \bigcup_{i=1}^3\mathcal E_i(2n+1)$$ is a bijection. To see that $\varphi_5$ is invertible, let $\mu\in \ds\bigcup_{i=1}^3\mathcal E_i(2n+1)$. \smallskip

If $\mu \in \mathcal E_1(2n+1)$, then $\varphi_5^{-1}(\mu)=(\mu_1-1, \mu_2, \ldots, \mu_{\ell(\mu)})$. 

If $\mu \in \mathcal E_2(2n+1)$, then $\varphi_5^{-1}(\mu)=\mu\setminus(1)$. 

If $\mu \in \mathcal E_3(2n+1)$, then $\varphi_5^{-1}(\mu)=(\mu_1/2, \mu_1/2, \mu_2+1, \ldots, \mu_{\ell(\mu)/2}+1)$, the partition obtained from $\mu$ by removing $\ell(\mu)/2$ parts equal to $1$, splitting the first part into two equal parts, and adding one to each of the remaining parts. 

Finally, we show that if $n\geq 3$, then $\overline{\mathcal P}^{\eu}_{\ou}(2n+1)\setminus \ds\bigcup_{i=1}^3\mathcal E_i(2n+1)\neq \emptyset$.

Let $n\geq 3$. If $2n+1=4k+1$ for $k\geq 2$, then $$(4(k-1), 1^5)\in \overline{\mathcal P}^{\eu}_{\ou}(2n+1)\setminus \ds\bigcup_{i=1}^3\mathcal E_i(2n+1).$$ If $2n+1=4k+3$ for $k\geq 1$, then $$(4k, 1^3)\in \overline{\mathcal P}^{\eu}_{\ou}(2n+1)\setminus \ds\bigcup_{i=1}^3\mathcal E_i(2n+1).$$

\end{proof}

\subsection{Examples}  To help clarify the map $\varphi_5$, we provide examples for the three cases.

\begin{example}
    Let $n = 12$. We give examples for partitions in $\overline{\mathcal P}_{\eu}^{\ou}(24)$ in Cases 1 and 2. 
\begin{itemize}
\item[\underline{Case 1:}] If $\lambda = (7^2, 5^2)$, then $\varphi_5(\lambda) = (8, 7, 5^2) \in \overline{\mathcal P}_{\ou}^{\eu}(25).$\smallskip 

\item[\underline{Case 2:}] If $\lambda = (12, 6^2)$, then $\varphi_5(\lambda) = (12, 6^2, 1) \in \overline{\mathcal P}_{\ou}^{\eu}(25).$\smallskip 
\end{itemize}
\end{example}

\begin{example}
    Let $n = 26$ and $\lambda = (11^2, 9^2, 8, 2^2)$. Then, $\lambda \in \overline{\mathcal P}_{\eu}^{\ou}(52)$ and $\varphi_5(\lambda) = (22, 8^2, 7, 1^8) \in \overline{\mathcal P}_{\ou}^{\eu}(53).$
\end{example}

    \section{Concluding Remarks} 

   We defined combinatorial injections to prove several inequalities between the numbers of different partitions of $n$ with parts separated by parity. However, we were unable to find combinatorial arguments for  the following inequalities in \eqref{chain}: 
    \begin{align}
\label{l-2}  p^{\eu}_{\od}(n)  & <p_{\eu}^{\od}(n);\\
\label{l-1}  p_{\eu}^{\od}(n) & 
<p_{\ed}^{\ou}(n);\\
\label{last} p_{\eu}^{\ou}(n) & <p^{\ed}_{\ou}(n).
\end{align}

We make some observations about inequality \eqref{last}. In particular, we show that it suffices to prove the inequality for even $n$. 

First, note that in \cite{BCN} the authors show that the injection $\lambda \mapsto \lambda\cup (1)$ proves combinatorially that $p^{\ed}_{\ou}(n)\leq p^{\ed}_{\ou}(n+1)$ for all $n\geq 0$. In addition, for $n\geq 2$, the partition $(n)$ is in $\mathcal P^{\ed}_{\ou}(n+1)$ but not in the image of the injection above. Thus, for $n\geq 2$, we have the strict inequality $p^{\ed}_{\ou}(n)< p^{\ed}_{\ou}(n+1).$ This inequality together with the result of the next proposition shows that it suffices to prove inequality \eqref{last} for even $n$.

\begin{proposition} \label{p2} Let $k\geq 0$. Then $$p_{\eu}^{\ou}(2k)=p_{\eu}^{\ou}(2k+1).$$
\end{proposition} 
\begin{proof}Let $k\geq 0$. We define a bijection $\psi: \mathcal P_{\eu}^{\ou}(2k)\to \mathcal P_{\eu}^{\ou}(2k+1)$.  

Let $\lambda\in \mathcal P_{\eu}^{\ou}(2k)$.

If $\lambda^e=\emptyset$, define $\psi(\lambda)=\lambda\cup(1)$. 

If $\lambda^e\neq\emptyset$, define $\psi(\lambda)$ to be the partition obtained from $\lambda$ by adding $1$ to the largest even part of $\lambda$. 

To see that $\psi$ is invertible, let $\mu\in\mathcal P_{\eu}^{\ou}(2k+1)$. Clearly, $\mu^o\neq \emptyset$.

If  $1\in \mu$ (which implies that $\mu^e=\emptyset$), then $\psi^{-1}(\mu)=\mu\setminus(1)$ is the partition obtained from $\mu$ by removing a part equal to $1$. 

If $1\not \in \mu$, then $\psi^{-1}(\mu)$ is the partition obtained from $\mu$ by subtracting $1$ from the smallest odd part of $\mu$. 
\end{proof}

While not needed for the argument that proving inequality \eqref{last} for even $n$ is sufficient, we also have the following strict inequality for $p_{\eu}^{\ou}(n)$. 

\begin{proposition} Let $k\geq 1$. Then \begin{equation} \label{o2} p_{\eu}^{\ou}(2k-1)< p_{\eu}^{\ou}(2k).\end{equation}
    \end{proposition}

    \begin{proof}  As noted in \cite{BCN}, the injection $$f(\eta)=f(\eta_1, \eta_2, \ldots, \eta_{\ell(\eta)}):=(\eta_1+2, \eta_2, \ldots, \eta_{\ell(\eta)})$$ shows that, for $n\geq 1$, we have $p_{\eu}^{\ou}(n)\leq p_{\eu}^{\ou}(n+2)$. Moreover, if $n\geq 1$, then the partition $(1^{n+2})$ is in $\mathcal P_{\eu}^{\ou}(n+2)$ but not in the image of $\mathcal P_{\eu}^{\ou}(n)$ under $f$. Hence, if $n\geq 1$, we have $p_{\eu}^{\ou}(n)< p_{\eu}^{\ou}(n+2)$.
    
    Inequality \eqref{o2} is clearly true for   $k=1$ since $p_{\eu}^{\ou}(1)=1$ and $p_{\eu}^{\ou}(2)=2$.
    If  $k\geq 2$, and $\psi$ is the bijection of Proposition \ref{p2}, the mapping $f\circ \psi^{-1}$ is an injection from $\mathcal P_{\eu}^{\ou}(2k-1)$ to $\mathcal P_{\eu}^{\ou}(2k)$. Since $f$ is not surjective, we have 
    $p_{\eu}^{\ou}(2k-1)< p_{\eu}^{\ou}(2k)$. 
     \end{proof}

We end this section by discussing further directions of study for partitions with parts separated by parity. 
 In addition to the modifiers $u$ (unrestricted) and $d$ (distinct), we could also include the modifier $nd$ to mean that the relevant parts (even or odd) are not all distinct.  For example, $\mathcal{P}^{\en}_{\ou}(n)$ is the set of partitions of $n$ with all even parts larger than all odd parts, there is at least one even part with multiplicity larger than 1, and odd parts unrestricted. From the inequalities \eqref{chain},  we can find several inequalities involving the modifier $nd$ (all holding for sufficiently large $n$).

First, consider the inequalities 
$p^{\eu}_{\od}(n)  <p_{\eu}^{\od}(n)
<p_{\ed}^{\ou}(n)<p_{\eu}^{\ou}(n) <p^{\eu}_{\ou}(n)$
and calculate several differences:

\begin{itemize}
\item $p^{\eu}_{\ou}(n) - p^{\eu}_{\od}(n) = p^{\eu}_{\on}(n),$
\smallskip

\item $p^{\ou}_{\eu}(n) - p^{\od}_{\eu}(n) = p^{\on}_{\eu}(n),$
\smallskip

\item $p^{\ou}_{\eu}(n) - p^{\ou}_{\ed}(n) = p^{\ou}_{\en}(n),$
\smallskip

\item $p^{\eu}_{\od}(n) - p^{\ed}_{\od}(n) = p^{\en}_{\od}(n).$
\end{itemize}
Therefore, $$p^{\on}_{\eu}(n)-p^{\ou}_{\en}(n) = p^{\ou}_{\eu}(n) - p^{\od}_{\eu}(n)-(p^{\ou}_{\eu}(n) - p^{\ou}_{\ed}(n))= p^{\ou}_{\ed}(n)-p^{\od}_{\eu}(n)>0$$ and 
\begin{align*}p^{\eu}_{\on}(n)- p^{\on}_{\eu}(n)& = p^{\eu}_{\ou}(n) - p^{\eu}_{\od}(n)-(p^{\ou}_{\eu}(n) - p^{\od}_{\eu}(n))\\ & = (p^{\eu}_{\ou}(n)-p^{\ou}_{\eu}(n) )+( p^{\od}_{\eu}(n)-p^{\eu}_{\od}(n)) >0.\end{align*}

This leads to the following inequalities. 
\begin{proposition} \label{p4} For $n$ sufficiently large, $$p^{\ou}_{\en}(n) < p^{\on}_{\eu}(n) < p^{\eu}_{\on}(n).$$
    
\end{proposition}

\smallskip

Next, consider the inequalities $p_{\ed}^{\od}(n) <p^{\ed}_{\od}(n) <p^{\eu}_{\od}(n)  <p_{\eu}^{\od}(n)
<p_{\ed}^{\ou}(n)$
and calculate several differences:

\begin{itemize}

\item $p^{\ou}_{\ed}(n) - p^{\od}_{\ed}(n) = p^{\on}_{\ed}(n),$
\smallskip

\item $p^{\od}_{\eu}(n) - p^{\od}_{\ed}(n)  = p^{\od}_{\en}(n),$
\smallskip

\item $p^{\eu}_{\od}(n) - p^{\ed}_{\od}(n)  = p^{\en}_{\od}(n),$
\smallskip

\item $p^{\ou}_{\eu}(n) - p^{\ou}_{\ed}(n) = p^{\ou}_{\en}(n).$
\end{itemize}
Therefore. \begin{align*}p^{\od}_{\en}(n)-p^{\en}_{\od}(n)& = p^{\od}_{\eu}(n) - p^{\od}_{\ed}(n) -(p^{\eu}_{\od}(n) - p^{\ed}_{\od}(n))\\ & =(p^{\ed}_{\od}(n)- p^{\od}_{\ed}(n))+ (p^{\od}_{\eu}(n)-p^{\eu}_{\od}(n))>0\end{align*} and 
$$p^{\on}_{\ed}(n)-p^{\od}_{\en}(n) =p^{\ou}_{\ed}(n) - p^{\od}_{\ed}(n)-(p^{\od}_{\eu}(n) - p^{\od}_{\ed}(n))= p^{\ou}_{\ed}(n)-p^{\od}_{\eu}(n)>0$$

This leads to the following inequalities. 
\begin{proposition} \label{p5} For $n$ sufficiently large, $$p^{\en}_{\od}(n) < p^{\od}_{\en}(n) < p^{\on}_{\ed}(n).$$
    
\end{proposition}

Moreover, to prove \eqref{l-1} combinatorially it suffices  to  find an injection from $\mathcal P^{\od}_{\en}(n)$ into $\mathcal P^{\on}_{\ed}(n)$.\smallskip

Lastly, consider the inequalities
$p^{\eu}_{\od}(n)  <p_{\eu}^{\od}(n)
<p_{\ed}^{\ou}(n)<p_{\eu}^{\ou}(n)  <p^{\eu}_{\ou}(n)$ and calculate several differences:

\begin{itemize}
\item $p^{\eu}_{\ou}(n) - p^{\eu}_{\od}(n) = p^{\eu}_{\on}(n),$
\smallskip

\item $p^{\ou}_{\eu}(n) - p^{\od}_{\eu}(n) = p^{\on}_{\eu}(n),$
\smallskip

\item $p^{\ou}_{\eu}(n) - p^{\ou}_{\ed}(n)  = p^{\ou}_{\en}(n),$
\smallskip

\item $p^{\eu}_{\ou}(n) - p^{\ed}_{\ou}(n)  = p^{\en}_{\ou}(n).$
\end{itemize}
Therefore $$p^{\on}_{\eu}(n)- p^{\ou}_{\en}(n)=p^{\ou}_{\eu}(n) - p^{\od}_{\eu}(n)-(p^{\ou}_{\eu}(n) - p^{\ou}_{\ed}(n))= p^{\ou}_{\ed}(n)-p^{\od}_{\eu}(n)>0$$
and
\begin{align*}p^{\eu}_{\on}(n)- p^{\on}_{\eu}(n)& = p^{\eu}_{\ou}(n) - p^{\eu}_{\od}(n)-( p^{\ou}_{\eu}(n) - p^{\od}_{\eu}(n))\\ & =(p^{\eu}_{\ou}(n)-p^{\ou}_{\eu}(n))+(p^{\od}_{\eu}(n)-p^{\eu}_{\od}(n)))>0\end{align*}

This leads to the following inequalities. 
\begin{proposition} \label{p6} For $n$ sufficiently large, $$p^{\ou}_{\en}(n) < p^{\on}_{\eu}(n) < p^{\eu}_{\on}(n).$$
    
\end{proposition}

It would be interesting to find combinatorial proofs for the inequalities in Propositions \ref{p4}--\ref{p6}.

\end{document}